\documentclass{amsart}
\usepackage{xy, enumerate, pstcol, amssymb, amsfonts, amsbsy, amsthm, amsmath, amscd, epsfig, latexsym, mathtools, stmaryrd, epic, eepic, eucal, multicol,multirow, fancyhdr, graphicx}

\DeclareMathAlphabet{\mathpzc}{OT1}{pzc}{m}{it}

\newenvironment{dem}{\begin{proof}[\bf Proof]}{\end{proof}}
\newtheorem{theorem}{\bf Theorem}[section]
\newtheorem*{maintheorem}{\bf Theorem 1}

\newtheorem{lemma}[theorem]{\bf Lemma}
\newtheorem{propos}[theorem]{\bf Proposition}

\newtheorem{claim}[theorem]{\bf Claim}

\theoremstyle{definition}
\newtheorem{defi}[theorem]{\bf Definition}
\newtheorem{rmk}[theorem]{\bf Remark}

\newcommand{\A}{\mathbb A}

\newcommand{\Ccal}{\mathcal C}

\newcommand{\G}{\mathbb G}

\newcommand{\Mcal}{\mathcal M}

\newcommand{\Ocal}{\mathcal O}
\newcommand{\Pro}{\mathbb P}

\newcommand{\Scal}{\mathcal S}

\newcommand{\Ucal}{\mathcal U}

\newcommand{\Z}{\mathbb Z}

\newcommand{\Isom}{\text{Isom}}

\newcommand{\GL}{\text{GL}}

\newcommand{\Spec}{\text{Spec}}

\newcommand{\pr}{\text{pr}}

\newcommand{\wt}{\widetilde}

\newcommand{\co}{\mathsf{c}}

\input xy
\xyoption{all}

\title{The integral Chow ring of $\Mcal_{0}(\Pro^r, 2)$}
\author{Renzo Cavalieri}
\author{Damiano Fulghesu}
\date{}

\address{Renzo Cavalieri, Department of Mathematics, Colorado State University, 1874 Campus Delivery, Fort Collins, CO, 80523-1874, U.S.A.}
\email{renzo@colostate.edu}
\address{Damiano Fulghesu, Department of Mathematics, Minnesota State University Moorhead, 1104 7th Ave South, Moorhead, MN 56563, U.S.A.}\email{fulghesu@mnstate.edu}


\begin{document}

\begin{abstract}
{We compute a presentation for the integral Chow rings of the moduli stacks of degree $2$ maps from smooth rational curves to projective space $\Pro^r$, as a quotient of a three-variable polynomial ring. The relations as $r$ varies have rich combinatorial structure: all non-trivial relations are encoded by two generating functions which are rational functions.}
\end{abstract}

\maketitle

\section{Introduction}

This article is the second step in the authors' program to compute the Chow rings with integral coefficients of the stacks $\Mcal_0(\Pro^r,d)$. We refer to the introduction of \cite{CF23} for history, context and relevance of the broader problem.  Here we focus on the work presented in this manuscript, and explain why it is a natural self-contained step in the program. Assume throughout characteristic different from $2$. We begin by stating the result.

\begin{maintheorem}[generating function form] \label{thm:mainthm}
For $r$ a positive integer, 
a presentation for the integral Chow ring of the stack $\Mcal_{0}(\Pro^r, 2)$ is given by:
\begin{equation*}
 A^\ast(\Mcal_{0}(\Pro^r, 2)) = \frac{\Z[T, c_2,c_3]}{(2c_3, (T^3 + c_2T + c_3)^{r+1}, \rho_{1,r}, \rho_{1,r+1}, (T^3+c_2T+c_3)\rho_{2,2r-2}, \rho_{2,2r}, \rho_{2,2r+2}) },
 \end{equation*}
 where $T$ is a degree one variable,  $c_i$ is a graded variable of degree $i$ and $\rho_{i,n}$  is the degree $n$ homogeneous term of the expansion at $0$ of the generating function $\mathcal{R}_i$, where

\begin{align}\label{eq:Rgf}
  \mathcal{R}_{1}   &= \frac{2}{(1-T)^2+c_2}\nonumber\\
 \mathcal{R}_{2}&=
 \frac{1}{(1-T^2)((1-T^2-c_2)^2+c_2) -c_2^2 +(T+c_3)c_3}
   \end{align}
\end{maintheorem}

The informal paraphrase of the result is: the integral Chow ring of the space of degree two rational maps to any projective space is presented as a quotient of a three-variable polynomial ring; the ideal of relations can be generated by (at most) seven relations, two of which are boring. The interesting relations can be extracted from the homogeneous coefficients of two rational functions in the three variables of the polynomial ring.

An alternative description of the main theorem  presents a generating set for the ideal of relations in closed form.

\begin{maintheorem}[closed form]
\[
 A^\ast(\Mcal_{0}(\Pro^r, 2)) \cong \frac{\Z[T,c_2,c_3]}{( 2c_3, (T^3+c_2T+c_3)^{r+1}, \{\alpha_{i,k}^r\}_{ i = 1,2  \
           k = 0,1,2
      })},   
      \]
      where
      \[
       \alpha_{1,k}^r =  \frac{2l_0^{k+1}(T+l_0)^{r+1}}{(l_1-l_0)(l_2-l_0)}+
    \frac{2l_1{^k+1}(T+l_1)^{r+1}}{(l_0-l_1)(l_2-l_1)}+
    \frac{2l_2^{k+1}(T+l_2)^{r+1}}{(l_0-l_2)(l_1-l_2)}.
      \]
      \[
       \alpha_{2,k}^r = {(-l_0)^k}\cdot\frac{(T+l_1)^{r+1}(T+l_2)^{r+1}}{(l_0-l_1)(l_0-l_2)}
    +
    {(-l_1)^k}\cdot\frac{(T+l_2)^{r+1}(T+l_0)^{r+1}}{(l_1-l_2)(l_1-l_0)}
    +{(-l_2)^k}\cdot\frac{(T+l_0)^{r+1}(T+l_1)^{r+1}}{(l_2-l_0)(l_2-l_1)}
      \]
\end{maintheorem}

The broad strategy for computing the presentation parallels \cite{CF23}:  present $\Mcal_0(\Pro^r,2)$ as a global quotient stack by a special group, thus identifying the Chow ring of the quotient with the group equivariant cohomology of the space being quotiented \cite{EG,Tot:CR};  show $\Mcal_0(\Pro^r,2)$ is an open substack of some  manageable ambient space, which provides a set of generators; construct an equivariant envelope for the complement and compute the pushforward of the Chow rings of the components of the envelope to obtain the relations. The localization computation on the equivariant envelope produces the presentation in the closed form version of Theorem \ref{thm:mainthm}. We verify the equivalence of the two forms of Theorem \ref{thm:mainthm} in Section \ref{sec:gfform}.

\vspace{0.5cm}

 The stacks $\Mcal_0(\Pro^r,d)$ admit a natural presentation as  global quotients by $\Pro \GL_2$: informally, considering the source of a map to be a projective line with homogeneous coordinates $(s:t)$, a point of $\Mcal_0(\Pro^r,d)$ can be represented by $(r+1)$ homogeneous polynomials in $s,t$ of degree $d$. All other representatives of the given moduli point are obtained as the orbit of the $\Pro\GL_2$ action reparameterizing the source curve, thus exhibiting $\Mcal_0(\Pro^r,d)$ as the global quotient of an open subset $U$ of projective space (where the homogeneous coordinates are the coefficients of the $(r+1)$ polynomials) by $\Pro\GL_2$.
 
For degree $d$ odd, a $\GL_2$ action on the affine cone (minus the zero section) over $U$ lifts the $\Pro\GL_2$ action and presents $\Mcal_0(\Pro^r,d)$ as a global quotient stack by the special group $\GL_2$ (\cite[Prop. 2.4]{CF23}). 

When $d$ is even such a lift does not exist. We  make use of a general construction from \cite{DL21} to exhibit $\Mcal_0(\Pro^r,d)$ as a global quotient stack by the special group $\GL_3$. Informally, for a moduli point $[f: C\to \Pro^r]\in \Mcal_0(\Pro^r,d)$, we view the source curve not as a parameterized $\Pro^1$, but as a smooth conic in $\Pro^2$. The map is then determined by $(r+1)$ sections of the restriction to the conic of the line bundle $\Ocal_{\Pro^2}(1)$, with all other representatives obtained as the orbit of a $\GL_3$ action lifting the standard linear action on $\Pro^2$.   

We give a slicker, albeit less intuitive, description of this construction: denoting by $\Scal\subset H^0(\Ocal_{\Pro^2}(2))\cong\mathbb{A}^6$ the locus of smooth conic equations, the map $\Scal\to [Spec(k)/\Pro\GL_2]$ is a $\GL_3$ torsor. The pull-back $Y$ of $[U/\Pro\GL_2]\cong \Mcal_0(\Pro^r,d)\to [Spec(k)/\Pro\GL_2]$ via this morphism is therefore also a $\GL_3$ torsor, implying $[Y/\GL_3]\cong \Mcal_0(\Pro^r,d)$.

We thus need to compute the $\GL_3$ equivariant Chow ring of $Y$. One can observe that $Y$ is an open subset of a projective bundle over $\Scal$, and we can therefore compute a presentation for its Chow ring by constructing an equivariant envelope for the complement of $Y$. For this task the case $d=2$ stands out with respect to all other even degrees for its simplicity, which led us to write its treatment as a standalone paper. In this case, there are two components of the envelope: one is a projective bundle over the universal conic, which is itself a projective bundle over $\Pro^2$. The second is isomorphic to a product of projective spaces. The localization computation, which can in both cases be performed relative to a morphism to $\Pro^2$, is then extremely simple and elegant.

We conclude  with a brief discussion of our observations and expectations for the remaining cases: when $d\geq 4$, the components of an equivariant envelope should be describable in terms of the geometry of either a subvariety of the absolute Hilbert scheme of points in $Hilb^n(\Pro^2)$, or the relative Hilbert scheme of the universal smooth conic over its base $\underline{Hilb}^n_{\Ucal/\Scal}$. Atiyah-Bott localization is not immediately  available in these cases, either because the space is non-proper or too singular, but we are exploring an approach through virtual intersection theory which we believe might allow for applying the localization theorem again; if this approach succeeds, then the combinatorics of localization and of trying to treat uniformly all even cases should make  the computations an interesting challenge for future work.

The paper is organized as follows: in Section \ref{sec:GL3} we provide some background on     $\GL_3$ counterparts (\cite{DL21}) to make this work more self-contained and set  notation. In Section \ref{sec:prese} we apply this construction to give a presentation of the stack $\Mcal_0(\Pro^r,2)$ as a global $\GL_3$ quotient stack. In Section \ref{Sec:env} we set up a $\GL_3$ equivariant envelope and compute the pushforward of the Chow rings of its components. In Section \ref{sec:gfform} we show how these computations give rise to the efficient generating function form of Theorem \ref{thm:mainthm}. We conclude by showing some explicit computations in low dimensions.

\subsection{Acknowledgments}
We would like to thank Henry Liu for interesting conversations related to this project, and Kavli IPMU, where these conversations happened, for their kind hospitality of the first author. We also would like to thank Zhengning Hu for valuable comments that helped us improve this manuscript.
The first author acknowledges support from the Simons Foundation travel grant MPS-TSM-00007937. 

\section{Preliminaries}
\subsection{$\GL_3$ counterparts of $\Pro \GL_2$ quotients}
\label{sec:GL3}

We summarize the $\GL_3$-counterpart construction of \cite{DL21}.

\begin{defi}[Definition 1.1, \cite{DL21}] Let $X$ be a scheme of finite type over $\Spec(k)$ endowed with a $\Pro \GL_2$ action. A $\GL_3$ counterpart of $X$ is a scheme $Y$ with a $\GL_3$ action and an isomorphism $\varphi$ of stacks:
\[
\varphi: [Y/\GL_3] \xrightarrow{\cong} [X/\Pro \GL_2].
\]
\end{defi}
Let $\Scal \subset W_{2,2}:=H^0(\Pro^2, \Ocal_{\Pro^2}(2))$ be the space of non-degenerate homogeneous quadratic forms in three variables. We have an isomorphism \cite[Proposition 1.5]{DL21}:
\[
\psi:[\Scal/\GL_3] \xrightarrow{\cong} [\Spec(k)/\Pro \GL_2],
\]
where the action of $\GL_3$ on $\Scal$ is defined by:
\begin{equation}\label{eq:action}
A \cdot P(x_0, x_1, x_2) := \det(A) P(A^{-1}(x_0, x_1, x_2)).
\end{equation}
In other words, $\Scal$ is a $\GL_3$ counterpart of the point. The isomorphism $\psi$ lifts to a morphism:
\[
\Scal \to [\Spec(k)/\Pro \GL_2],
\]
which makes $\Scal$ a $\GL_3$-torsor over 
the classifying space $[\Spec(k)/\Pro \GL_2]$.

Let $X$ be a $\Pro \GL_2$-scheme. The structure morphism $X \to \Spec(k)$ induces a morphism of quotient stacks:
\[
[X/\Pro \GL_2] \to [\Spec(k)/\Pro \GL_2],
\]
which is representable. We  consider the following cartesian diagram:
\[
\xymatrix{
  Y \ar[r] \ar[d] & [X/\Pro \GL_2] \ar[d] \\
  \Scal \ar[r] & [\Spec(k)/\Pro \GL_2]
}
\]
where $Y$ (which is a scheme thanks to the representability of the morphism $[X/\Pro \GL_2] \to [\Spec(k)/\Pro \GL_2]$) is a $\GL_3$-torsor over $[X/\Pro \GL_2]$ and, consequently we get an isomorphism:
\[
[Y/\GL_3] \xrightarrow{\cong} [X/\Pro \GL_2].
\]
In other words, the cartesian product of $[X/\Pro \GL_2]$ and $\Scal$ over $[\Spec(k)/\Pro \GL_2]$ is a $\GL_3$ counterpart of $X$.

A $\GL_3$ counterpart $Y$ of $X$ yields an isomorphism:
\[
A^*_{\GL_3}(Y) \cong A^*_{\Pro \GL_2}(X);
\]
the $\GL_3$ equivariant Chow ring of $Y$ therefore computes the Chow ring of $[X/\Pro\GL_2]$.

\subsection{{The Chow ring of a subset of a projective bundle}}

We introduce some notions that will be useful later on. Consider a $\Scal$-space $Y$ consisting of the restriction to $\Scal$ of the projectivization of a $\GL_3$-equivariant open subset $U$ of a $\GL_3$-equivariant vector bundle $F$ over $W_{2,2}$, as in the following commutative diagram with cartesian square:
\begin{equation}
\xymatrix{
  Y \cong \Pro(U|_{\Scal}) \ar[r] \ar[d] & \Pro(U = F\smallsetminus Z) \ar@{^{(}->}[r] \ar[d] & \Pro(F) \ar[dl]\\
  \Scal \ar[r] & W_{2,2} &
}.
\end{equation}

For any space $X$, if $i: \Ucal \to X$ is a $\GL_3$-equivariant open inclusion and $\pi: \Ucal^\co \to X$  is the $\GL_3$-equivariant closed inclusion of the complement of $\Ucal$ in $X$, we have an isomorphism:
\begin{equation} \label{eq:excision}
A^*_{\GL_3} \left( \Ucal \right) \cong \frac{A^*_{\GL_3} \left( X \right)}{\pi_* \left( A^*_{\GL_3} \left( \Ucal^\co \right) \right)}.
\end{equation}
Using \eqref{eq:excision}, we can compute the equivariant Chow ring of $\Pro(U|_{\Scal})$ in two steps. First we determine the equivariant Chow ring of $\Pro(F|_{\Scal})$:
\[
A^*_{\GL_3}\left( \Pro \left( F|_{\Scal} \right) \right) \cong \frac{A^*_{\GL_3} \left( \Pro(F) \right)}{\pi_* \left( A^*_{\GL_3} \left( \Pro \left( F|_{\Scal^\co} \right) \right) \right)},
\]
where we identify the complement of $\Pro{\left( F|_{\Scal} \right)}$ in $\Pro(F)$ with $\Pro \left( F|_{\Scal^\co} \right)$.

Then we use $A^*_{\GL_3}\left( \Pro \left( F|_{\Scal} \right) \right)$ to determine the Chow ring of $\Pro(U|_{\Scal})$:
\[
A^*_{\GL_3}{\Pro(U|_{\Scal})} \cong \frac{A^*_{\GL_3} \left( \Pro \left( F|_{\Scal} \right) \right)}{\pi_* \left( A^*_{\GL_3} \left(\Pro(U|_{\Scal})^\co) \right) \right)}.
\]
where $\Pro(U|_{\Scal})^\co$ is the complement in $\Pro(F|_{\Scal})$.

In the following lemma, we give a presentation for the intersection ring  $\left( A^*_{\GL_3} \left( \Pro \left( F|_{\Scal} \right) \right) \right)$.

\begin{lemma}
For $F$ a $\GL_3$ equivariant vector bundle over $W_{2,2}$, we  have an isomorphism: 
\[
 \left(A^*_{\GL_3} \left( \Pro(F|_{\Scal})\right) \right) \cong \frac{A^*_{\GL_3}\left( \Spec(k) \right)[T]}{\left( \pi_*\left(A^*_{\GL_3}\left( \Scal^\co \right) \right), P(T) \right)}
\]
for a suitable hyperplane class $T$ in $A^*_{\GL_3} \left( \Pro(F) \right)$ and a polynomial $P(T)$ in $\A^*_{\GL_3}(W_{2,2}) \cong \A^*_{\GL_3} (\Spec(k))$.
\end{lemma}

\begin{dem} Since $F$ is a vector bundle over the vector space $W_{2,2}$, we have an isomorphism:
\[
A^*_{\GL_3} \left( \Pro(F) \right) \cong \frac{A^*_{\GL_3}(W_{2,2})[T]}{(P(T))}
\]
where $T$ is the first Chern class of the $\Ocal(1)$ bundle of $\Pro(F)$ and $P(T)$ is the total equivariant Chern polynomial of $F$ over $W_{2,2}$.

In a similar way, since $F|_{\Scal^\co}$ is a vector bundle over $\Scal^\co$, we have an isomorphism:
\[
A^*_{\GL_3} \left( \Pro(F|_{\Scal^\co}) \right) \cong \frac{A^*_{\GL_3}(\Scal^\co)[T]}{(P'(T))}
\]
where $T$ is the first Chern class of the $\Ocal(1)$ bundle of $\Pro(F|_{\Scal^\co})$ and $P'(T)$ is the total equivariant Chern polynomial of $F|_{\Scal^\co}$ over $\Scal^\co$. The class $T$ in $A^*_{\GL_3} \left( \Pro(F|_{\Scal^\co}) \right)$ is the pull back of the class $T$ in $A^*_{\GL_3} \left( \Pro(F) \right)$, and we keep the same name for simplicity.

Now, the image of the push-forward:
\[
\pi_*: A^*_{\GL_3} \left( \Pro(F|_{\Scal^\co}) \right) \to A^*_{\GL_3} \left( \Pro(F) \right)
\]
is the ideal in $A^*_{\GL_3} \left( \Pro(F) \right)$ generated by $\pi_* \left(A^*_{\GL_3}\left( \Scal^\co \right)\right)$, formally seen as a subalgebra of $A^*_{\GL_3} \left( \Pro(F) \right)$, and the push-forward of the powers of the class $T$. However, since the class $T$ is the pullback of the class $T$ in $A^*_{\GL_3} \left( \Pro(F) \right)$, thanks to the pull-push formula, the push-forward of any power of the class $T$ is in the ideal generated by $\pi_*(1)$ which is already included in $\pi_* \left(A^*_{\GL_3}\left( \Scal^\co \right)\right)$. In conclusion, we have the isomorphism:
\begin{equation}\label{eq:int.amb}
 \left(A^*_{\GL_3} \left( \Pro(F|_{\Scal})\right) \right) \cong \frac{A^*_{\GL_3}\left( \Spec(k) \right)[T]}{\left( \pi_*\left(A^*_{\GL_3}\left( \Scal^\co \right) \right), P(T) \right)}.
\end{equation}
\end{dem}

\subsection{The equivariant Chow ring of $\Scal$.}
\label{sec:ecrs}

We compute the Chow ring of the quotient stack $[\Scal/\GL_3]$ which is isomorphic to the classifying space $[\Spec(k)/\Pro \GL_2]$. This result is well known;  we present the computations here to clarify our notation and to explain some computational tools that will be used throughout the paper.

Let $\Delta$ be the maximal torus of $\GL_3$ represented by diagonal matrices and let $V$ be the standard representation of $\GL_3$. The total character of the $\Delta$-module $V^\vee$ (the dual of $V$) can be expressed as a sum of linearly independent characters $\lambda_0$, $\lambda_1$, and $\lambda_2$ and we get $A^*_{\Delta} = \Z[l_0, l_1, l_2]$ where $l_i := c_1(\lambda_i)$. The Weyl group $S_3$ acts on $A^*_{\Delta}$ by permuting the classes $l_i$ and, consequently, $A^*_{\GL_3}(V) \cong A^*_{\GL_3}(\Spec(k))=\Z[c_1, c_2, c_3]$ where:

\begin{eqnarray*}
c_1 &=& -l_0 - l_1 - l_2,\\
c_2 &=& l_0l_1 + l_0l_2 + l_1l_2,\\
c_3 &=& -l_0l_1l_2.
\end{eqnarray*}
\begin{rmk}
    We spell out our convention for projectivization, which will be needed throughout the paper.  We chose
\(l_0,l_1,l_2\) to be the Chern roots of \(V^\vee\).  If
\(p_i\in \mathbb{P}(V)\) is the torus fixed point corresponding to the
\(i\)-th coordinate line, and \(p_i^\vee\in\mathbb{P}(V^\vee)\) is the
dual fixed point corresponding to the $i$-th coordinate hyperplane, then the
canonically linearized hyperplane classes $H, H^\vee$ satisfy
\[
H|_{p_i}:=\left.c_1^\Delta\bigl(\mathcal{O}_{\mathbb{P}(V)}(1)\bigr)\right|_{p_i}
=l_i,
\qquad
H^\vee|_{p_i^\vee}:=\left.c_1^\Delta\bigl(\mathcal{O}_{\mathbb{P}(V^\vee)}(1)\bigr)\right|_{p_i^\vee}
=-l_i.
\]
Consequently,
\[
e_\Delta\bigl(T_{p_i}\mathbb{P}(V)\bigr)
=\prod_{j\neq i}(l_i-l_j),
\qquad
e_\Delta\bigl(T_{p_i^\vee}\mathbb{P}(V^\vee)\bigr)
=\prod_{j\neq i}(l_j-l_i).
\]
\end{rmk}

It is well known (see e.g. \cite[Proposition 6]{EG}) that one may identify the $\GL_3$-equivariant intersection ring of any space with the $S_3$-invariant part of the $\Delta$-equivariant intersection ring. We therefore set out to compute the torus equivariant intersection ring of $\Scal$.

Let $P(x_0,x_1,x_2)$ be an element of $\Scal$, that is to say a rank 3 homogeneous quadratic form in the projective coordinates $x_0,x_1,x_2$ of $\Pro(V)$. We observe that the action  defined in \eqref{eq:action} restricts to $\Scal$.

Let $\Scal^\co$ be the complement of $\Scal$ in $W_{2,2}$, then we have:
\[
A^*_{\GL_3}(\Scal) \cong \frac{\Z[c_1,c_2,c_3]}{\left( \pi_*(A^*_{\GL_3}(\Scal^\co))\right)}.
\]

To determine $\pi_*(A^*_{\GL_3}(\Scal^\co))$, we consider the universal curve over $W_{2,2}$:
\[
V_{2,2} \subset W_{2,2} \times \Pro^2(V),
\]
where $V$ is the standard representation of $\GL_3$. The hypersurface $V_{2,2}$ is given by the bihomogeneous equation of bidegree $(1,2)$:
\[
M(a,x):= a_{2,0,0}x_0^2 + a_{0,2,0}x_1^2 + a_{0,0,2}x_2^2 + a_{1,1,0}x_0 x_1 + a_{1,0,1}x_0 x_2 + a_{0,1,1}x_1 x_2.
\]

Let $H$ be the hyperplane class in $\Pro^2(V)$. For simplicity, we also call $H$ the pullback of $H$ to $W_{2,2} \times \Pro^2(V)$. In both cases, the class $H$ satisfies the relation:
\[
(H-l_0)(H-l_1)(H-l_2) = H^3 + c_1H^2 + c_2H + c_3=0.
\]

Let $\widetilde{\Scal^\co}$ be the locus of singular points on the degenerate conics. More precisely, $\widetilde{\Scal^\co}$ is the subscheme of $V_{2,2}$ given by the equations $\left \{ M_{x_i}(a,x)=0\right \}_{i=0,1,2}$, where $M_{x_i}(a,x)$ is the partial derivative of $M(a,x)$ with respect to the variable $x_i$. A straightforward dimensional argument shows that $\widetilde{\Scal^\co}$ is a complete intersection. This means that the equivariant class of $\widetilde{\Scal^\co}$ in $W_{2,2} \times \Pro^2(V)$ is the product of the hyperplane classes $J_i$ associated to the equations $M_{x_i}(a,x)=0$.

Now, let $A:=\left[ \begin{array}{ccc} t_0 & 0 & 0 \\ 0 & t_1 & 0\\ 0 & 0 & t_2 \end{array}\right]$ be an element of $\Delta$. By applying $A$ according to the action \eqref{eq:action} to the torus invariant hyperplane (for example) $J_0 = \{M_{x_0} = 0\}$, we get (see \cite[Lemma 2.4]{EF09}):
\[
\left [ J_0 \right] = H + c_1(t_1 \cdot t_2) = H - l_1 - l_2.
\]
Similarly, we get:
\begin{eqnarray*}
\, [J_1] &=& H - l_0 - l_2, \\
\, [J_2] &=& H - l_0 - l_1.
\end{eqnarray*}
Therefore, the class of $\widetilde{\Scal^\co}$ in $W_{2,2} \times \Pro^2(V)$ is:

\[
[J_0] \cdot [J_1] \cdot [J_2] = (H-l_1-l_2)(H -l_0 - l_2)(H - l_0 - l_1)
\]
which after reducing mod $(H-l_0)(H-l_1)(H-l_2)$ becomes:
\[
c_1H^2 + c_1^2H + (c_1c_2 - 2c_3).
\]
Consider the following commutative diagram:
\begin{equation*}
\xymatrix{
  \widetilde{\Scal^\co} \ar[r]^{\hspace{-1cm} \pi} \ar[d]_{\text{pr}_{1}} & W_{2,2} \times \Pro^2(V) \ar[d]^{\text{pr}_{1}} \\
  \Scal^\co \ar[r]^{\pi} & W_{2,2}
}
\end{equation*}
The map $pr_1$ over $\Scal^\co$ is surjective. Moreover, also the pushforward is surjective, since $\widetilde{\Scal^\co}$ can be seen as a Chow envelope with two components:
\begin{itemize}
\item one that maps 1-1 to the image corresponding to the locus of degenerate conics that are union of two distinct lines with the intersection acting as the unique singular point;
\item one that corresponds to double lines and the restricted morphism is a projective bundle.
\end{itemize}
On both components, the push-forward is surjective, therefore, the whole pushforward is surjective.
By commutativity of the diagram, we have the identity:
\[
\pi_* \left( \Scal^\co \right) = \text{pr}_{1*} \left( \pi_* \left( \widetilde{\Scal^\co} \right) \right).
\]
Following the proof of \cite[Proposition 4.2]{FV18},
we obtain that the ideal $\pi_* \left( \widetilde{\Scal^\co} \right)$ is generated by the class $\left[ \widetilde{\Scal^\co} \right]$.
Arguing as in \cite[Theorem 4.5]{FV18},  the ideal $\text{pr}_{1*} \left( \pi_* \left( \widetilde{\Scal^\co} \right) \right)$ is generated by the coefficients of $\left[ \widetilde{\Scal^\co} \right]$ seen as a polynomial in $H$. In conclusion, $\text{pr}_{1*}(\Scal^\co)$ is generated by $c_1, c_1^2$
and
$c_1c_2 - 2c_3$.
%
%

Therefore $\pi_*(A^*_{\GL_3}(\Scal^\co))= (c_1, 2c_3)$ and we have the isomorphism:
\[
A^*_{\GL_3}(\Scal) \cong \frac{\Z[c_1,c_2,c_3]}{(c_1, 2c_3)}.
\]

\begin{rmk}
The use of the universal curve was helpful to determine the ideal $\pi_* \left( \Scal^\co \right)$ since we were able to lift up $\Scal^\co$ to a complete intersection $\widetilde{\Scal^\co}$, and the push-forward of its intersection ring is easier to compute.
\end{rmk}

We can now make \eqref{eq:int.amb} more explicit:
\begin{equation}
 \left(A^*_{\GL_3} \left( \Pro(F|_{\Scal})\right) \right) \cong \frac{A^*_{\GL_3}\left( \Spec(k) \right)[T]}{\left( c_1, 2c_3, P(T) \right)}.
\end{equation}

\section{A Presentation of the Stack $\Mcal_{0}(\Pro^r, 2)$}
\label{sec:prese}
 Let $V_{2,2} \subset W_{2,2} \times \Pro^2$ be the universal curve over $W_{2,2}$ and $\Ccal$ its restriction on $\Scal$. We have the following diagram:
\begin{equation}
\xymatrix{
\Ccal \ar[d]_{\pr_1}  \ar[r]^{\subseteq} & V_{2,2} \ar[d]^{\pr_1}  \ar[r]^{\pr_2} & \Pro^2\\
\Scal \ar[r]^{\subseteq} & W_{2,2}
}
\end{equation}
Define the following sheaves:
\begin{eqnarray}
E^{r} &:=& \pr_{2}^* \left( \Ocal_{\Pro^2}(1)^{\oplus r+1} \right) \otimes \Ocal_{V_{2,2}}
\nonumber \\
F^{r} &:=& \pr_{1*} \left( E^{r}\right)
\end{eqnarray}
respectively over $V_{2,2}$ and $W_{2,2}$, which are locally free (vector bundles) of ranks $r+1$ and $3r+3$. 

\begin{rmk}
  We provide a more intuitive description of these objects.  $W_{2,2}$ is the space of equations for plane conics and $V_{2,2}$ the affine cone over the universal conic. Given a conic equation $\varphi_C$ and a point $p\in \Pro^2$ with $\varphi_C(p) = 0$, a point in the fiber $E^r_{(\varphi_C,p)}$ is just an ordered $(r+1)-$tuple of elements in the fiber $\Ocal_{\Pro^2}(1)_p$; a point in the fiber of $F^r_{\varphi_C}$ consists of (the restriction to the conic  $C$ of) $(r+1)$ homogeneous linear forms $\vec{\sigma} = (\sigma_0, \ldots, \sigma_r)$.  
\end{rmk}

There is a natural evaluation morphism of sheaves:
\begin{align*}
  \pr_1^*(F^{r}) &\to E^{r}  \\
  (\varphi_C,p, \vec{\sigma})&\mapsto (\varphi_C,p, \vec{\sigma}(p)),
\end{align*}
whose kernel we denote $K^{r}$.

Denote $U^{r} := \left. (F^{r} \smallsetminus \pr_1(K^{r})) \right |_{\Scal} $ the open subset of $\left. F^{r} \right |_{\Scal}$ consisting  of points $(\varphi_C, \vec{\sigma})$ such that the $(r+1)$ linear forms do not have a common intersection on the conic $C$.

\begin{theorem}
There is an isomorphism:
\[
\Mcal_{0}(\Pro^r, 2) \cong [\Pro(U^{r})/\GL_3].
\]
\end{theorem}

\begin{dem}
Let $\mathcal{W}$ be the vector space of degree 2 forms in 2 variables.
From \cite{CF23}, we have the following isomorphism:
\[
\Mcal_{0}(\Pro^r, 2) \cong \left[ X/ \Pro\GL_2 \right],
\]
where $X$ is the open locus in $\Pro(\mathcal{W}^{\oplus{r+1}})$ of $r+1$ forms that do not have a common factor. By \cite{DL21}, a $\GL_3$ counterpart of $X$ is the cartesian product:

\begin{equation}
\xymatrix{
  Y \ar[r] \ar[d] & [X/\Pro \GL_2] \ar[d]^{\alpha} \\
  \Scal \ar[r]^(0.25){\beta} & [\Spec(k)/\Pro \GL_2]
}
\end{equation}

Therefore we need to show that $\Pro(U^r)$ is isomorphic to $Y$. We define  two morphisms $\psi$ and $\pi$:
\begin{equation} \label{diag:stackcomm}
\xymatrix{
  \Pro(U^r) \ar[r]^{\psi} \ar[d]_{\pi} & [X/\Pro \GL_2] \ar[d]^{\alpha} \\
  \Scal \ar[r]^(0.25){\beta} & [\Spec(k)/\Pro \GL_2]
}.
\end{equation}
The morphism $\pi$ is the natural bundle projection. The stack morphism $\psi$ is uniquely determined by an object in $[X/\Pro \GL_2]$ via the following cartesian diagram:
\[
\xymatrix{
  P_{\psi} \ar[r] \ar[d] & X \ar[d] \\
  \Pro(U^r) \ar[r]^(0.4){\psi}  & [X/\Pro \GL_2]
}
\]
where $P_{\psi}$ is a $\Pro\GL_2$ torsor over $\Pro(U^r)$ with an equivariant map to $X$. A natural choice for $P_{\psi}$ is the pullback $\pi^*(P_{\beta})$ of the $\Pro\GL_2$ torsor $P_{\beta} \to \Scal$ (seen as an object in $[\Spec(k)/\Pro \GL_2]$) defined by the morphism $\beta$ via the following cartesian diagram:
\[
\xymatrix{
  P_{\beta} \ar[r] \ar[d] & \Spec(k) \ar[d] \\
  \Scal \ar[r]^(0.25){\beta} & [\Spec(k)/\Pro \GL_2],
}
\]

Note that the fibers $P_{\beta} \to \Scal$ on each smooth conic $C$ define the isomorphism group $\Isom_{k}(\Pro^1, C)$.

We conclude by checking the following properties of diagram \eqref{diag:stackcomm}.
\begin{itemize}
\item {\bf Stack-commutative:}  For every scheme morphism $T \xrightarrow{f} \Pro(U^r)$, the images $\alpha (\psi(f))$ and $\beta(\pi(f))$ in $[\Spec(k)/\Pro \GL_2]$ are isomorphic.  By construction, the images $\alpha (\psi(f))$ and $\beta(\pi(f))$ are both pullbacks on $T$ of $P_{\beta}$ and therefore they are isomorphic.
\item {\bf Universal property of fiber products:}  For every stack-commutative diagram:
\begin{equation}
\label{diag:up}
\xymatrix{
  W \ar[r]^(0.4){\theta} \ar[d]_{\rho} & [X/\Pro \GL_2] \ar[d]^{\alpha} \\
  \Scal \ar[r]^(0.25){\beta} & [\Spec(k)/\Pro \GL_2],
}
\end{equation}
we want to show that there exists a morphism (unique up to isomorphisms) $W \to \Pro(U^r)$ that makes all the resulting diagrams commute.
By the commutativity of diagram \eqref{diag:up}, the pullback $P_{\theta} \to W$ of $X \to [X/\Pro \GL_2]$ via $\theta$ is isomorphic to the pullback of $P_{\beta} \to \Scal$ via $\rho$.

In particular, we get $\Pro GL_2$-equivariant morphisms from $P_{\theta}$ to $X$ and to $P_{\beta}$. In other words, a point on $P_{\theta}$ defines a conic $C$, an isomorphism $\gamma:  \Pro^1\to C$  and a projective equivalence class $[\vec{Q}]$ of $(r+1)$ (ordered) quadratic forms on $\Pro^1$.
Via the isomorphism $\gamma^*(\Ocal_{\Pro^2}(1))\cong \Ocal_{\Pro^1}(2)$, $[\vec{Q}]$  defines a projective equivalence class of $(r+1)$ linear forms on 
$\Pro^2$. This construction defines a $\Pro GL_2$-equivariant morphism $P_{\theta} \to P_{\psi}$. Such morphism induces a unique morphism on the quotient spaces $P_{\theta}/ \Pro GL_2 \to P_{\psi}/\Pro GL_2$ which gives  the requested (unique up to isomorphisms) morphism:
\[
P_{\theta}/ \Pro GL_2 \cong W \to \Pro(U^r) \cong P_{\psi}/\Pro GL_2.
\]

\end{itemize}

\end{dem}

The natural $\G_m$-action on $W_{2,2}$ which induces the standard multiplication of a quadratic form by a non-zero scalar lifts to the bundle $F^{r} \smallsetminus \pr_1(K^{r})$ by acting trivially on the homogeneous coordinates of the fibers. After taking the restriction of $F^{r} \smallsetminus \pr_1(K^{r})$ over $W_{2,2} \smallsetminus {0}$ and the quotient by $\G_m$, we get a bundle $\check U^{r}$ over $\Pro(W_{2,2}) \cong \Pro^5$. The $\GL_3$ action on $F^{r} \smallsetminus \pr_1(K^{r})$ naturally descends to $\check U^{r}$

We can reduce the computation of $A^\ast(\Mcal_{0}(\Pro^r, 2))$ to the computation of $A^\ast_{\GL_3}(\Pro(\check U^r))$ as shown in the following Lemma.

\begin{lemma}\label{lem:comp} We have the following isomorphism:
\[
    A^\ast(\Mcal_{0}(\Pro^r, 2)) \cong A^\ast_{\GL_3} (\Pro( U^r)) \cong \left(\frac{A^\ast_{\GL_3}(\Pro(\check U^r))}{(Q, c_1, 2c_3)} \right)
    \]
where $Q$ is the pullback to $\check U^{r}$ of the hyperplane class in $\Pro(W_{2,2})$.
\end{lemma}
\begin{dem}
 Generalizing the results from Section \ref{sec:ecrs}, we have:
\begin{equation}\label{eqn.red.to.closed}
A^\ast_{\GL_3} \left( \Pro(\check U^r) \right) \cong \frac{A^\ast_{\GL_3} \left( \Pro(F^{r} \smallsetminus \pr_1(K^{r})) \right)}{(c_1, 2c_3)}.
\end{equation}
Since $\check U^r$ is obtained from $F^{r} \smallsetminus \pr_1(K^{r})$ by projectivization of the base, we follow the argument of \cite[Lemma 3.2]{EF09}. By considering the action \eqref{eq:action} on the base, we get:
\begin{equation}\label{eqn.red.to.proj.base}
A^\ast_{\GL_3} \left( \Pro(F^{r} \smallsetminus \pr_1(K^{r})) \right) \cong \frac{A^\ast_{\GL_3}(\Pro(\check U^r))}{(Q-c_1)}.
\end{equation}
We conclude by combining the ring isomorphisms (\ref{eqn.red.to.closed}) and (\ref{eqn.red.to.proj.base}).
\end{dem}

\section{Computation of the presentation}
\label{Sec:env}

In this section we  compute the $\GL_3$-equivariant cohomology of $\Pro(\check{U}^r)$.
We begin with some steps aimed at simplifying the computations. Denote by $\Delta\leq \GL_3$ the maximal torus consisting of diagonal matrices. It is well known (see e.g. \cite[Proposition 6]{EG}) that one may identify the $\GL_3$-equivariant cohomology of any space with the $S_3$-invariant part of the $\Delta$-equivariant cohomology. We therefore set out to compute the torus equivariant cohomology $A^\ast_{\Delta}(\Pro(\check U^r)).$
We identify the torus equivariant parameters with the Chern roots of the dual of the standard representation (see Section \ref{sec:ecrs}), so $A^\ast_\Delta(\Spec(k)) = \mathbb{Z}[l_0,l_1,l_2]$ and the classes $c_i$ are the elementary symmetric functions evaluated at $-l_0,-l_1,-l_2$.

The second simplifying step consists in observing that in the action \eqref{eq:action}, $W_{2,2}$ is made into a representation by inducing the natural action from the dual of the standard representation and then tensoring with equivariant line bundle ``$\det$". Since $\det$ has first Chern class  $c_1$ and in Lemma \ref{lem:comp}  we are eventually setting $c_1 = 0$, we may as well ignore the tensoring by $\det$ in the action on $W_{2,2}$. This way, anywhere the variable $x_i$ appears (whether as a coordinate of $\Pro^2$, or as part of a monomial in $\Pro(W_{2,2})$), it has weight $l_i$. 

Consider the open immersion:
\[
\Pro(\check{U}^{r}) \stackrel{i}{\hookrightarrow} \left. \Pro(\check{F}^{r} \right |_{\Pro(\Scal)})
\]
Denote by ${Z}$ the Zariski closure of the complement $\left.\Pro(\check{F}^{r} \right |_{\Pro(\Scal)}) \smallsetminus \Pro(\check{U}^{r})$ in $\Pro(\check{F}^{r})$. By the excision sequence we have
\begin{equation} \label{eq:excis}
A^\ast_{\Delta}(\Pro(\check U^r))\cong   
\frac{A^\ast_{\Delta}(\left. \Pro(\check{F}^{r} \right |_{\Pro(\Scal)}))}{i_\ast (A^\ast_{\Delta}({Z}))}
\end{equation}

In Sections \ref{sec:Z1},\ref{sec:Z2}, we construct an {\it equivariant envelope} (\cite[Definition 18.3]{Ful},\cite[Section 2.6]{EG}) 

\[
\xymatrix{
\widetilde{Z} = Z_1 \sqcup Z_2  \ar[rr]^{\widetilde{\pi} =\pi_1\sqcup \pi_2}& &  Z,}
\]
i.e.  a proper, torus equivariant map such that for every subvariety $V$ of $Z$, there is a subvariety $\wt{V}$ of $\wt{Z}$ with the morphism $\widetilde{\pi}$ mapping $\wt{V}$ birationally onto $V$.
Informally, points in $Z_1$ correspond to a conic, a point $p$ on the conic, and $r+1$ equations of lines through $p$; for points of $Z_2$ we have a conic, two points on the conic and $r+1$ equations of lines through both points.

By \cite[Theorem 3.1]{CF23},
\begin{equation} \label{eq:env}
   i_\ast( A^\ast_{\Delta}({Z}))= (i\circ\tilde{\pi} )_\ast (A^\ast_{\Delta}(\widetilde{Z})).
\end{equation}

Combining \eqref{eq:excis}, \eqref{eq:env}, in order to compute $A^\ast_{\Delta}(\Pro(\check U^r))$ we must compute the Chow ring of the ambient space $\left. \Pro(\check{F}^{r} \right |_{\Pro(\Scal)})$, and then quotient by the ideal of relations obtained from the pushforward of the Chow ring of the two components $Z_1, Z_2$ of the envelope. The next few subsections are dedicated to these computations.

\subsection{The Chow ring of the ambient}
In this section we give a presentation for the Chow ring of $\left. \Pro(\check{F}^{r} \right |_{\Pro(\Scal)})$. We recall the standard presentation of the equivariant cohomology of projective space $\Pro^5 = \Pro (W_{2,2})$, see e.g. \cite[Section 27.1]{Mirror-symmetry-book}:
\begin{equation}
A^\ast_{\Delta}(\Pro^5) \cong \frac{\Z[Q, c_1,c_2, c_3]}{(P_2(Q))},
\end{equation}
where $Q =  c_1^{eq}(\Ocal_{\Pro^5}(1))$ is the (equivariant, canonically linearized) hyperplane class and for any $d$ we define the polynomial 

\begin{equation}
    \label{eq:pd}
P_d(Q) =  \prod_{\begin{array}{c} 0\le i,j,k \le d\\
i+j+k = d\end{array}}
\left(
Q+i l_0+jl_1+kl_2
\right).
\end{equation}
Then $P_2(Q)$ is the relation arising from the equivariant Chern class of the bundle $W_{2,2}\to B\Delta$.

\begin{propos}\label{prop:amb}
\[
A^\ast_{\Delta}(\left. \Pro(\check{F}^{r} \right |_{\Pro(\Scal)})) \cong \frac{\Z[T, c_2, c_3]}{(2c_3, (T^3+c_2T+c_3)^{r+1})}.
\]   
\end{propos}

\begin{proof}
    The Chow ring of $\Pro (\check{F}^{r})\stackrel{\pi}{\to} \Pro^5$ has a natural presentation coming from the projective bundle formula \cite[Theorem 3.3]{Ful}:
\begin{equation} \label{eq:prespbF}
A^\ast_{\Delta}(\Pro (\check F^{r})) \cong \frac{\Z[Q,T, c_1, c_2, c_3]}{(P_2(Q),R(T))},
\end{equation}
where $T = c_1^{eq}(\Ocal_{\Pro (F^{r})}(1))$ is the relative hyperplane class, and $R(T)$ is the polynomial homogenizing the Chern class of $\check{F}^{r}$. The vector bundle $\check{F}^r$ is a trivial (but not equivariantly trivial) bundle consisting of the direct sum of $(r+1)$ copies of $H^0(\Pro^2, \Ocal_{\Pro^2}(1))$, i.e. of the dual of the standard representation of $\GL_3$. It follows that
\[
R(T) = (T^3-c_1T^2+c_2T-c_3)^{r+1}.
\]

Restricting the projective bundle to the part over the locus $\Pro(\Scal) \subset \Pro^5$, one obtains the relations computed in Section \ref{sec:ecrs}: $Q = c_1$ and  $c_1 = 2c_3 = 0$. Plugging these into \eqref{eq:prespbF} one obtains:
\[
A^\ast_{\GL_3}(\left. \Pro(\check{F}^{r} \right |_{\Pro(\Scal)})) \cong \frac{\Z[T, c_2, c_3]}{(2c_3, P_2(0), (T^3+c_2T+c_3)^{r+1})}.
\]
The proof of the proposition is  concluded by observing that $P_2(0)$ contains a factor of $8c_3$ coming from the weights of the quadratic forms $x_0^2,x_1^2, x_2^2$ which are $2l_0,  2l_1$ and $ 2l_2$; $P_2(0)$ is therefore a redundant relation as the polynomial is already contained in the ideal generated by $2c_3$.
\end{proof}

\subsection{The first component $Z_1$ of the envelope}
\label{sec:Z1}
We  construct a projective variety $Z_1$ together with a birational, regular, $\Delta$-equivariant morphism $\pi_1: Z_1 \to {Z}$; this gives the first component of the envelope. 

We define $Z_1$ as  a subset of $\Pro(\pr_1^\ast (\check F^{r}))$:
\[
Z_1:= \{(C, p\in C, \sigma_0, \ldots , \sigma_r) | \sigma_i(p) = 0\ \  \mbox{for all $i$}\},
\]
where $C$ denotes a (not necessarily smooth) conic, the pair $(C,p)$ denotes a point in the universal conic $C_{2,2}$, and $\sigma_i$ is a linear form restricted to the conic, i.e. a global section of $\Ocal_{\Pro^2}(1)_{|C}$. The map $\pi_1: Z_1 \to Z$ is the pull-back of the natural morphism that forgets the point $p$, i.e. the universal family from the universal conic $C_{2,2}$.

We observe that $Z_1$ may be described as a projective bundle over a projective bundle over $\Pro^2$. First, define $K_2$  to be the rank $5$ bundle obtained as the kernel of the  evaluation sequence:
\[
0\to K_2 \to H^0(\Pro^2, \Ocal_{\Pro^2}(2))\times \Pro^2 \stackrel{ev_2}{\to} \Ocal_{\Pro^2}(2) \to 0.
\]
The universal conic $C_{2,2}$ is naturally identified with $\Pro(K_2)$. Analogously, we consider the bundle $K_1$ defined as the kernel of the evaluation sequence:
\[
0\to K_1 \to H^0(\Pro^2, \Ocal_{\Pro^2}(1))\times \Pro^2 \stackrel{ev_1}{\to} \Ocal_{\Pro^2}(1) \to 0.
\]
The affine cone $\widehat{Z}_1$ is naturally identified with $(r+1)$ copies of the pullback of $K_1$, so:
\[
{Z}_1 = \Pro(\pr_2^\ast K_1^{\oplus (r+1)}).
\]
Presenting $Z_1$ as a projective bundle over a projective bundle over $\Pro^2$ allows us to identify a set of generators for $A^\ast_{\Delta}(Z_1)$:
\begin{itemize}
    \item the equivariant parameters $l_0, l_1, l_2$ generating $A^\ast_{\Delta}(\Spec(k))$;
    \item the (pullback of the) equivariant hyperplane class of $\Pro^2$, $H = c_1^{eq}(\Ocal_{\Pro^2}^{can}(1))$;
    \item the relative (equivariant) hyperplane class of the projective bundle $\Pro(K_2) \to \Pro^2$, which we denote by $Q$;
    \item the relative (equivariant) hyperplane class of the projective bundle $\Pro(\pr_2^\ast K_1^{\oplus (r+1)}) \to \Pro(K_2)$, which we denote by $T$.
\end{itemize}
The names of the generators have been purposefully chosen: one may show via some standard diagram chasing that the classes $Q,T$ are the pull-back via $\pi_1$ of the homonymous classes in \eqref{eq:prespbF}. It  follows by the projection formula that $\pi_{1 \ \ast}(A^\ast_{\Delta}(Z_1))$ is generated as an ideal by the pushforwards of the classes $1, H, H^2$. We observe that these classes are pulled-back via the morphism $\Pi_1: Z_1 \to \Pro^2$ obtained by composing the two successive projective bundle projections.

By the Atiyah-Bott localization isomorphism ({\cite[Theorem 2]{EGL}}), 
 the class $H^k$ in $A^*_\Delta(\Pro^2)\otimes\Z[\frac{1}{l_0},\frac{1}{l_1}, \frac{1}{l_2}]$ may be expressed in the fixed point basis as:
	\begin{equation}\label{eq:ABloc} H^k = \sum_{j = 0}^2 \frac{i_{p_j}^*(H^k)}{c_{top}^\Delta(T_{p_{j}}\Pro^2)} [p_j] =  \frac{l_0^k}{(l_0-l_1)(l_0-l_2)}[p_0] +
    \frac{l_1^k}{(l_1-l_0)(l_1-l_2)}[p_1] +
    \frac{l_2^k}{(l_2-l_0)(l_2-l_1)}[p_2], 
    \end{equation}
where $[p_j]$ denotes the equivariant cohomology class of the torus fixed point where the $j$-th coordinate is non-zero.

Denote by $L_j = \Pi_1^{-1}(p_j)$ the inverse image in $Z_1$ of the $j$-th fixed point of $\Pro^2$. The map $\pi_{1| L_j}$ is an isomorphism onto its image, which consists of the torus invariant subvariety of $\Pro (\check F^{r})$ defined by the vanishing of:
\begin{itemize}
    \item the homogeneous coordinate of $W_{2,2}$ corresponding to the conic $x_j^2$;
    \item the homogeneous coordinates in $\Pro^{3r+2}$ that are congruent to $j$ modulo $3$, corresponding to choosing $r+1$ times the linear form $x_j$.
\end{itemize}
As a consequence, we have
\begin{equation} \label{eq:pjloc}
    \pi_{1 \ast}(\Pi_1^\ast ([p_j])) = (Q+2l_j)(T+l_j)^{r+1}.
\end{equation}
Combining \eqref{eq:ABloc},\eqref{eq:pjloc} we obtain:
\begin{equation}
    \label{eq:Z1rel}
    \pi_{1 \ast}(H^k) =  \frac{l_0^k(Q+2l_0)(T+l_0)^{r+1}}{(l_0-l_1)(l_0-l_2)}+
    \frac{l_1^k(Q+2l_1)(T+l_1)^{r+1}}{(l_1-l_0)(l_1-l_2)}+
    \frac{l_2^k(Q+2l_2)(T+l_2)^{r+1}}{(l_2-l_0)(l_2-l_1)}.
\end{equation}

It is immediate to see that \eqref{eq:Z1rel} is a symmetric rational function in $l_0,l_1,l_2$. In order to see that it is in fact a polynomial expression with integer coefficients it suffices to show that it does not have poles along the loci $l_i = l_j$. Picking, without loss of generality, $l_0 = l_1$ one sees that the third summand is clearly regular along this locus. As for the first two summands, consider their product by $(l_1-l_0)$, thus clearing the factor in the denominator; it is immediate to see that  setting $l_1=l_0$ the sum of the two terms vanishes, meaning that the original two summands added to a function regular along $l_1=l_0$. By symmetry, it follows that \eqref{eq:Z1rel} are  homogeneous polynomials of degree $r+k$, symmetric in the $l_i$'s.

\subsection{The second component  $Z_2$ of the envelope}
\label{sec:Z2}

We define $Z_2$ to be
\[
Z_2 = \Pro(W_{2,2})\times (\Pro^2)^\vee \times \Pro^r,
\]
with the action of the torus $\Delta$ scaling linearly the  $x_i$'s in the first  factor and the homogeneous coordinates of the second, and acting trivially on the third factor $\Pro^r$.
The  equivariant map $\pi_2:Z_2\to Z$ is given by:
\[
(C,(a_0:a_1:a_2),(b_0:\ldots:b_r))\mapsto (C, (b_0\cdot(a_0x_0+a_1x_1+a_2x_2):\ldots: b_r\cdot(a_0x_0+a_1x_1+a_2x_2))).
\]
Generically, the image $\pi_2(Z_2)$ parameterizes a choice of a conic and $r+1$ linear forms vanishing at two points: generically the two points are distinct and determine a unique line, and therefore the linear forms are all scalar multiples of each other.

The $\Delta$-equivariant Chow ring of $Z_2$ is generated by the equivariant classes $c_1,c_2, c_3$ and the equivariant hyperplane classes of the factors $Q, H, T_2$, where the relationship with the pullbacks of the hyperplane classes  via $\pi_2$ are: $\pi_2^\ast(Q) = Q$ and $\pi_2^\ast (T) = H+T_2$.
By the projection formula, the ideal generated by the image of $\pi_{2 \ast}$ is generated by $\pi_{2\ast}(H^k)$ for $k = 0,1,2$.

Consider the projection function $\Pi_2: Z_2\to (\Pro^2)^\vee$. As in Section \ref{sec:Z1}, the class of $H^k$ may be written in the fixed point basis as:

\begin{equation}\label{eq:ABlocZ2} H^k = \sum_{j = 0}^2 \frac{i_{p^\vee_j}^*(H^k)}{c_{top}^\Delta(T_{p^\vee_{j}}(\Pro^{2})^\vee)}[p_j^\vee] =  \frac{(-l_0)^k}{(l_1-l_0)(l_2-l_0)}[p_0^\vee] +
    \frac{(-l_1)^k}{(l_0-l_1)(l_2-l_1)}[p_1^\vee] +
    \frac{(-l_2)^k}{(l_0-l_2)(l_1-l_2)}[p_2^\vee], 
    \end{equation}
where $[p_j]^\vee$ is the class of the linear form $x_j$ in the dual plane.

The restriction of $\pi_2$ to $\Pi_2^{-1}(p_j^\vee)$ maps isomorphically onto its image, consisting of (the inverse image via the projection to $\Pro^{3r+2}$ of) the linear subspace defined by the vanishing of the homogeneous coordinates 
not congruent to $j$ modulo $3$. It follows that:

 \begin{equation}
    \label{eq:Z2rel}
    \pi_{2 \ast}(H^k) =  \frac{(-l_0)^k(T+l_1)^{r+1}(T+l_2)^{r+1}}{(l_1-l_0)(l_2-l_0)} + \frac{(-l_1)^k(T+l_0)^{r+1}(T+l_2)^{r+1}}{(l_0-l_1)(l_2-l_1)}+
    \frac{(-l_2)^k(T+l_0)^{r+1}(T+l_1)^{r+1}}{(l_0-l_2)(l_1-l_2)}.
\end{equation}
 
Expression \eqref{eq:Z2rel} is clearly symmetric, as any element of $S_3$ permutes the three summands. For any $i\not=j$, it is immediate to see that $(l_i-l_j)\pi_{2 \ast}(H^k)$ vanishes when setting $l_i =l_j$, implying that $\pi_{2 \ast}(H^k)$ is in fact a degree $2r+k$ homogeneous polynomial which is symmetric in the $l_i$'s. 

\subsection{A presentation for $A^\ast(\mathcal{M}_{0}(\Pro^r,2))$ }
\label{sec:pre}
We have all the ingredients to exhibit explicitly the presentation from Lemma \ref{lem:comp}.
Denote by $\alpha^r_{i,k} = \pi_{i \ast}(H^k)_{|Q = 0}$ the relations obtained from the $i$-th component of the envelope,  as computed in  \eqref{eq:Z1rel},\eqref{eq:Z2rel}. From \eqref{eq:excis}
 and Proposition \ref{prop:amb} we obtain:
 \begin{equation}
     \label{eq:presentation}
      A^\ast(\Mcal_{0}(\Pro^r, 2)) \cong \frac{\Z[T,c_2,c_3]}{( 2c_3, (T^3+c_2T+c_3)^{r+1}, \{\alpha_{i,k}^r\}_{ i = 1,2  \
           k = 0,1,2
      })},
 \end{equation}
proving the closed form version of Theorem \ref{thm:mainthm}.

\section{Proof of generating function form of Theorem \ref{thm:mainthm}}

\label{sec:gfform}

In this section we perform algebraic manipulations that allow us to describe the relations from the presentation in Section \ref{sec:pre} as coefficients of appropriate generating functions, and thus give a proof of the generating function version of Theorem \ref{thm:mainthm}. As a starting point we define  generating functions for  relations of the form $\alpha_{i,k}^r$, as $r$ varies.
Define 
\begin{equation}
    \mathcal{A}_{i,k}:=\sum_{r=0}^\infty \alpha_{i,k}^r.
\end{equation}

For the generating functions $\mathcal{R}_j$ from \eqref{eq:Rgf}, we let $\rho_{j,n}$ denote the degree $n$ part of the formal geometric series expansion of  $\mathcal{R}_j$.

We first consider the relations coming from the first component of the envelope.  

\begin{claim} \label{claim:firstenv}
We have the following equality of ideals in $\Z[T,c_2,c_3]/(2c_3)$:
\[
(\alpha_{1,0}^r, \alpha_{1,1}^r,\alpha_{1,2}^r) = (\rho_{1,r},\rho_{1,r+1}).
\]
\end{claim}
\begin{proof}
We explicitly compute the generating function $\mathcal{A}_{1,0}$. Setting $Q=0$ in \eqref{eq:Z1rel}, and recognizing that each summand gives rise to a geometric series, we obtain:
\begin{align}
  \mathcal{A}_{1,0} 
  &=  \frac{2l_0}{(l_0-l_1)(l_0-l_2)(1-T-l_0)}+
\frac{2l_1}{(l_1-l_0)(l_1-l_2)(1-T-l_1)}+
    \frac{2l_2}{(l_2-l_0)(l_2-l_1)(1-T-l_2)}
  \nonumber \\
  &=\frac{2(1-T)}{(1-T-l_0)(1-T-l_1)(1-T-l_2)} = \frac{2(1-T)}{(1-T)^3+(1-T)c_2+c_3}.
\end{align}
Viewing the denominator of $\mathcal{A}_{1,0}$ as having the shape $1+P(T,c_2,c_3)$, with $P$ a polynomial vanishing at $(0,0,0)$, the relations are  obtained by expanding the denominator as a geometric series. Because of the factor of $2$ in the numerator, any term containing $c_3$ appears with an even coefficient and therefore vanishes in the ring $\Z[T,c_2,c_3]/(2c_3)$.
It follows that replacing $P(T,c_2, c_3)$ by $P(T,c_2,0)$ does not change any of the relations. Setting $c_3 = 0$, it is immediate to observe that $\mathcal{A}_{1,0}$ reduces to $\mathcal{R}_1$. We have therefore shown that, for every $r$, 
\begin{equation} \label{eq:a0rr}
\alpha_
{1,0}^r = \rho_{1,r},
\end{equation}
where we recall that the second lower index for the $\rho$'s indicates the degree of the relation rather than the power of the hyperplane class pushed forward.

Resumming in an analogous fashion all the relations $\alpha_{1,1}^r$  one obtains:
\begin{equation}
    \mathcal{A}_{1,1} = (1-T)\mathcal{A}_{1,0}-2,
\end{equation}
from which it follows that
\begin{equation}\label{eq:a0rr+1}
    \alpha_{1,1}^r = \alpha_{1,0}^{r+1}-T\alpha_{1,0}^r =\rho_{1,r+1}-T\rho_{1,r} 
\end{equation}
Finally, resumming over $r$ the relations $\alpha_{1,2}^r$ one obtains:
\begin{equation}
 \mathcal{A}_{1,2} = -c_2\mathcal{A}_{1,0},   
\end{equation}
implying
\begin{equation}\label{eq:a0rc2}
    \alpha_{1,2}^r = -c_2\alpha_{1,0}^{r} = -c_2\rho_{1,r}.
\end{equation}
The claim immediately follows from \eqref{eq:a0rr},\eqref{eq:a0rr+1},\eqref{eq:a0rc2}.
\end{proof}
 We proceed analogously with the relations from the second envelope.
\begin{claim}\label{claim:secondenv}
We have the following equality of ideals in $\Z[T,c_2,c_3]/(2c_3)$:
\[
(\alpha_{2,0}^r, \alpha_{2,1}^r,\alpha_{2,2}^r) = (\rho_{2,2r},\rho_{2,2r-2}(T^3+c_2T+c_3),\rho_{2,2r+2}).
\]
\end{claim}
\begin{proof}
We adopt a similar strategy as in the previous claim. We begin by summing over all $r$ the relations $\alpha_{2,0}^r$ from \eqref{eq:Z2rel}, to obtain
\begin{align}
    \mathcal{A}_{2,0} &= \frac{1}{(l_1 - l_0)(l_2 - l_0)(1 - (T + l_1)(T + l_2))} +\frac{1}{(l_0 - l_1)(l_2 - l_1)(1 - (T + l_0)(T + l_2))} \nonumber\\
     &+\frac{1}{(l_0 - l_2)(l_1 - l_2)(1 - (T + l_0)(T + l_1))} \nonumber
    \\
     & = -\frac{1}{((T^2 + (l_0 + l_1)T + l_0l_1 - 1)(T^2 + (l_0 + l_2)T + l_0l_2 - 1)(T^2 + (l_1 + l_2)T + l_1l_2 - 1))}\nonumber\\
     & =  \frac{1}{(1-T^2)((1-T^2-c_2)^2+c_2) -c_2^2 +(T+c_3)c_3} = \mathcal{R}_2.
\end{align}
This computation shows that for any $r$
\begin{equation}\label{eq:seconenvforstrel}
    \alpha_{2,0}^r = \rho_{2,2r}
\end{equation}
Analogously resumming all the relations $\alpha_{2,k}^r$, for $k = 1, 2$, one obtains: 
\begin{align}\label{eq:A20}
  \mathcal{A}_{2,1}&=
-\frac{T^3+(c_2-1)T+c_3}{(1-T^2)((1-T^2-c_2)^2+c_2) -c_2^2 +(T+c_3)c_3} = -(T^3+c_2T+c_3 -T)\mathcal{R}_2,\nonumber\\
  \mathcal{A}_{2,2} +1 &=
 \frac{T^4 + (c_2 - 2)T^2 + c_3T - c_2 + 1}{(1-T^2)((1-T^2-c_2)^2+c_2) -c_2^2 +(T+c_3)c_3} = (T^4+c_2T^2+c_3T -2T^2-c_2+1)\mathcal{R}_2.
 \end{align}

At the level of coefficients, \eqref{eq:A20} gives:
\begin{align}\label{eq:seconenvrelequiv}
    \alpha_{2,1}^r = T\rho_{2,2r}-(T^3+c_2T+c_3)\rho_{2,2r-2},\nonumber\\
     \alpha_{2,2}^r = -T\alpha_{2,1}^r-(T^2+c_2)\rho_{2,2r}+\rho_{2,2r+2}.
\end{align}
Equations \eqref{eq:seconenvforstrel},\eqref{eq:seconenvrelequiv} together immediately establish the Claim.
\end{proof}

Claims \ref{claim:firstenv} and \ref{claim:secondenv} allow to substitute the relations $\alpha_{i,k}^r$'s with the corresponding (multiples of) $\rho$'s in presentation \eqref{eq:presentation}, thus concluding the proof of Theorem \ref{thm:mainthm}.

\section{Low dimensional examples}

We conclude the article with a few explicit examples of the presentations obtained from Theorem 
\ref{thm:mainthm}. In Table \ref{tab:conics} we write generators for the ideal of relations for the presentation of $A^\ast(\Mcal_0(\Pro^r,2))$, for $r = 1,2,3$.
We collect in the table a redundant set of generators coming from the equivariant envelope. It is sufficient to consider only the relations from one of the two columns: the first column expresses the pushforwards of $H^k$, thus giving the relations more naturally arising from the envelope; the second column gives the relations in terms of the coefficients of the two generating functions $\mathcal{R}_1, \mathcal{R}_2$, i.e. the relations arising from the most efficient generating function packaging. 
\pagebreak

\begin{table}[t]
$\Mcal_0(\Pro^1,2)$
\vspace{0.3cm}

    \centering
\begin{tabular}{|l|l|}
\hline
\multicolumn{2}{|c|}{Ambient relations: $2c_3, (T^3+c_2T+c_3)^2$}\\
\hline
\multicolumn{2}{|c|}{Classes from $Z_1$}\\
\hline
$\alpha_{1,0}^1=4T$ &  $\rho_{1,1} =4T$ 
 \\
$\alpha_{1,1}^1 = 2T^2 - 2c_2$ & $\rho_{1,2} =
6T^2 - 2c_2
$\\
\hline
\multicolumn{2}{|c|}{Classes from $Z_2$}  \\
\hline
$\alpha_{2,0}^1=3T^2+c_2$&
$\rho_{2,2} =3T^2+c_2$
\\
$\alpha_{2,1}^1=2T^3+c_3$& $\rho_{2,0}(T^3+c_2T+c_3) =T^3+c_2T+c_3$\\
$\alpha_{2,2}^1=T^4-c_2T^2$&$\rho_{2,4} = 6T^4 + 3c_2T^2 + c_3T + c_2^2$\\
\hline
\end{tabular}
\vspace{0.3cm}

$\Mcal_0(\Pro^2,2)$
\vspace{0.3cm}

    \centering
\begin{tabular}{|l|l|}
\hline
\multicolumn{2}{|c|}{Ambient relations: $2c_3, (T^3+c_2T+c_3)^3$}\\
\hline
\multicolumn{2}{|c|}{Classes from $Z_1$}\\
\hline
$\alpha_{1,0}^2=6T^2 - 2c_2$ &  $\rho_{1,2} =6T^2 - 2c_2$ 
 \\
$\alpha_{1,1}^2 = 2T^3 - 6c_2T$ & $\rho_{1,3} =
8T^3 - 8c_2T
$\\
\hline
\multicolumn{2}{|c|}{Classes from $Z_2$}  \\
\hline
$\alpha_{2,0}^2=6T^4 + 3c_2T^2+c_3T+c_2^2$&
$\rho_{2,4} =6T^4 + 3c_2T^2+c_3T+c_2^2$
\\
$\alpha_{2,1}^2=3T^5 -c_2T^3 + c_2c_3$& $\rho_{2,2}(T^3+c_2T+c_3) =3T^5 +4c_2T^3+c_3T^2+c_2^2T+c_2c_3$\\
$\alpha_{2,2}^2=T^6 -3c_2T^4 + c_3T^3 + c_3^2$&$\rho_{2,6} = 10T^6+5c_2T^4+4c_2^2T^2+c_3^2+c_2^3$\\
\hline
\end{tabular}

\vspace{0.3cm}

$\Mcal_0(\Pro^3,2)$
\vspace{0.3cm}

    \centering
\begin{tabular}{|l|l|}
\hline
\multicolumn{2}{|c|}{Ambient relations: $2c_3, (T^3+c_2T+c_3)^4$}\\
\hline
\multicolumn{2}{|c|}{Classes from $Z_1$}\\
\hline
$\alpha_{1,0}^3=8T^3 - 8c_2T$ &  $\rho_{1,3} =8T^3 - 8c_2T$ 
 \\
$\alpha_{1,1}^3 =2T^4 - 12c_2T^2 + 2c_2^2 $ & $\rho_{1,4} =10T^4 - 20c_2T^2 + 2c_2^2

$\\
\hline
\multicolumn{2}{|c|}{Classes from $Z_2$}  \\
\hline
$\alpha_{2,0}^3=10T^6 + 5c_2T^4  + 4c_2^2T^2  + c_2^3 + c_3^2$&
$\rho_{2,6} =10T^6 + 5c_2T^4  + 4c_2^2T^2  + c_2^3 + c_3^2$
\\
$\alpha_{2,1}^3=4T^7 - 4c_2T^5 + c_3T^4 + c_2^2c_3$& $\rho_{2,4}(T^3+c_2T+c_3) =6T^7 + 9c_2T^5 + c_3T^4 + 4c_2^2T^3  + c_2^3T + c_3^2T + c_2^2c_3$\\
$\alpha_{2,2}^3=T^8 - 6c_2T^6  + c_2^2T^4  + c_2c_3^2$&$\rho_{2,8} = 15T^8 + 5c_2T^6 +c_3T^5 + 10c_2^2T^4  + 5c_2^3T^2 + c_3^2T^2 +c_2^2c_3T + c_2^4 $\\
\hline
\end{tabular}
    \caption{Generators for the ideal of relations in the presentation for $A^\ast(\Mcal_0(\Pro^r,2))$, for $r = 1,2,3$ from Theorem \ref{thm:mainthm}.}
    \label{tab:conics}
\end{table}

\bibliographystyle{alpha} 
\bibliography{Kontsevich.bib} 

\end{document}